\documentclass[12pt]{amsart}

\usepackage{amsmath,amssymb,amsthm}
\usepackage{amsfonts}
\usepackage[mathscr]{eucal}

\newtheorem{theorem}{Theorem}[section]

\theoremstyle{definition}
\newtheorem{defi}[theorem]{Definition}

\theoremstyle{remark}
\newtheorem{remark}[theorem]{Remark}

\numberwithin{equation}{section}

\date{January 27 2009}

\newcommand{\RR}[1]{\mathbb{#1}}

\newcommand{\rd}{{\mathbb R^d}}

\newcommand{\rr}{{\mathbb R}}

\begin{document}

\title{\bf Stochastic solutions of Conformable fractional Cauchy problems}

\author{Y\" ucel \c Cenesiz}
\address{Y\" ucel \c Cenesiz, Department of Mathematics, Selcuk University, Konya, Turkey}
\email{}

\author{Ali Kurt}
\address{Ali Kurt, Department of Mathematics, Mustafa Kemal University,Hatay , Turkey}
\email{}

\author{Erkan Nane}
\address{Erkan Nane, Department of Mathematics and Statistics,
Auburn University, Auburn, AL 36849}
\email{ezn0001@auburn.edu}

\begin{abstract}
In this paper we  give stochastic solutions of conformable fractional Cauchy problems. The stochastic solutions are obtained by running the processes corresponding to Cauchy problems  with a nonlinear deterministic clock.
\end{abstract}


\maketitle
\section{Introduction}
Scientists have paid great attention to fractional calculus which is known as the differentiation and integration  of arbitrary order. In the last few decades, great amount of work  was carried out  on fractional calculus in different fields of engineering and science extensively
with variety of applications. While these studies have been cariied out, scientists used different definitions of fractional derivative and integral such as Gr\"unwald-Lethnikov, Reisz-Fischer, Caputo, Riemann-Liouville, modified Riemann-Liouville and etc. But almost all of these derivatives have some kind of flaws. For instance, the Riemann-Liouville fractional derivative of a constant is not zero, the Riemann-Liouville derivative and Caputo derivative do not obey the Leibnitz rule and chain rule. The Riemann-Liouville derivative and Caputo do not satisfy the known formula of the derivative of the quotient of two functions. The Caputo definition assumes that the function $f$ is smooth (at least absolutely continuous). Recently a new type of fractional derivative  called conformable fractional derivative that overcomes these flaws has been introduced by R. Khalil et al. \cite{rkhalil}.\\
\begin{defi} Let $f:\left[ 0,\infty \right) \rightarrow
\mathbb{R}
$ be a function. The $\alpha ^{th}$ order \textit{"conformable fractional
derivative" \ }of $f$ is defined by,
\[
T_{\alpha }(f)(t)=\lim_{\varepsilon \rightarrow 0}\frac{f(t+\varepsilon
t^{1-\alpha })-f(t)}{\varepsilon },
\]

for all $t>0, \alpha \in (0,1).$ If $f$ is $\alpha $-differentiable in some $%
(0,a),a>0$ and $\lim\limits_{t\rightarrow 0^{+}}f^{(\alpha )}(t)$ exists
then define $f^{(\alpha )}(0)=\lim\limits_{t\rightarrow 0^{+}}f^{(\alpha
)}(t)$ and the \textit{"conformable fractional integral"} of a function $f$ starting from $a\geq0$ is defined as:
\[I_{\alpha }^{a}(f)(t)=\int\limits_{a}^{t}\frac{f(x)}{x^{1-\alpha }}dx\] where the integral is the usual Riemann improper integral, and ${\alpha }\in(0,1]$.
\end{defi}
The following theorem gives some properties which are satisfied by the conformable fractional derivative.
\begin{theorem}[\cite{rkhalil}]
Let $\alpha \in (0,1]$ and suppose $f,g$ are $\alpha $-differentiable at point $t>0$.
Then
\begin{enumerate}
\item $T_{\alpha }(cf+dg)=cT_{\alpha }(f)+cT_{\alpha }(g)$ for all $a,b\in
\mathbb{R}
$.

\item $T_{\alpha }(t^{p})=pt^{p-\alpha }~$for all $p\in
\mathbb{R}
.$

\item $T_{\alpha }(\lambda )=0$ for all constant functions $f(t)=\lambda .$

\item $T_{\alpha }(fg)=fT_{\alpha }(g)+gT_{\alpha }(f).$

\item $T_{\alpha }\left( \frac{f}{g}\right) =\frac{gT_{\alpha
}(f)-fT_{\alpha }(g)}{g^{2}}.$

\item In  addition, if   $f$  differentiable, then $T_{\alpha
}(f)(t)=t^{1-\alpha }\frac{df}{dt}.$
\end{enumerate}
\end{theorem}
Because of its applicability, and  effectiveness  scientists studied conformable derivative in various fields. For instance; Abdeljawad \cite{abdel} has presented basic principles of calculus such as chain rule, integration by parts and etc.; Ghanbari and  Gholami \cite{gholami} used  the so called fractional conformable operators for establishing fractional Hamiltonian systems; Iyiola and Nwaeze \cite{iyi} proved some results on conformable fractional derivatives and fractional integrals, also they apply the d'Alambert approach to a conformable fractional differential equation; Anderson \cite{ander} employed the conformable derivative to formulate several boundary value problems with three or four conformable derivatives, including those with conjugate, right-focal, and Lidstone conditions; Zhao \textit{et al.} \cite{zao} investigate a new concept of conformable delta fractional derivative which has the identity operator on time scales; Pospisil and Skripkova \cite{mic} proved Sturm's separation and Sturm's comparison theorems for differential equations involving a conformable fractional derivative of order $0<\alpha \leq 1$ using local properties of conformable fractional derivative. Hence, it is  worthwile to work on this new area.

In this paper, we will study the stochastic solution of equations with conformable time derivative where the space operators may correspond to fractional Brownian motion, or  a Levy process, or a general semigroup in a Banach space, or a  process killed upon exiting a bounded domain in $\mathbb{R}^d$.

Self-similar processes arise naturally in limit theorems of random walks and other stochastic processes, and they have
been applied to model various phenomena in a wide range of scientific areas including physics, finance, telecommunications, turbulence,
image processing, hydrology and economics. The most important example of self-similar processes is fractional Brownian motion (fBm)
which is a stationary and  centered Gaussian process $B^H = {W H (t), t \geq 0}$ with $B^H (0) = 0$ and covariance function
\begin{equation}\label{cov-fBM}
E [B^ H (s)B^H (t)]= \frac12(|s|^2H + |t|^2H - |s - t|^2H).
\end{equation}
where $H \in (0, 1)$ is a constant. By using \eqref{cov-fBM} one can verify that $B^H$ is self-similar with index $H$ (i.e., for all constants $c > 0$,
the processes $\{B^H (ct), t \geq 0\}$ and $\{c^H B^ H (t), t\geq 0\}$ have the same finite-dimensional distributions) and has stationary
increments. When $H = 1/2$, $B^H$ is a  Brownian motion, which will be written as $B$.
This process was studied and popularized
in one-dimension by Mandelbrot \cite{mandel}.

One of the main results in this paper is the following theorem.
\begin{theorem}\label{fbm-ctfd}
Let $\alpha=2H\in (0,1)$, and  $B^H(t)$ be  Fractional Brownian motion  of index $H\in (0,1/2)$ started at 0. Then, the unique solution of  the equation with conformable time derivative
\begin{equation}
\begin{split}
T_{2H} u(t,x)&=H \Delta_x u(t,x), t>0, x\in \rd\\
u(0,x)&=f(x), \ x\in \rd
\end{split}
\end{equation}
is given by
\begin{equation}
u(t,x) = E[f(x+B^H(t))])
\end{equation}
\end{theorem}

\begin{remark}O'Malley \textit{et al.} \cite{malley} studied a multiple dimensional
extension of fBm. O'Malley and Cushman \cite{malley2} developed a nonlinear extension of classical fBm,
which is fBm run with a nonlinear clock.
In the last theorem and the other results in this paper  we observe that replacing the first time derivative with a conformable time fractional  time derivative corresponds to replacing the time in the stochastic solution of  a Cauchy problem  with  a nonlinear clock. Please see Meerschaert and Sheffler \cite{limitCTRW} for the corresponding time change when we work with a time fractional  Cauchy problem with a Caputo fractional derivative.

\end{remark}

Our paper is organized as follows. In the next section we give  some preliminaries about semigroups on a Banach space and L\'evy semigroups. In section 3, we state main theorems and give the proof of Theorem \ref{fbm-ctfd}.

\section{Preliminaries on L\'evy processes}

 Given a Banach space and a
bounded continuous semigroup $P(t)$ on that space with generator
$L_x$, it is well known that $u(t,x)=P(t)f(x)$ is the unique
solution to the abstract Cauchy problem
\begin{equation}\label{ACP}
\frac{\partial}{\partial t}u(t,x) =L_x u(t,x); \quad u(0,x) = f(x)
\end{equation}
for any $f$ in the domain of $L_x$; see, for example,
\cite{ABHN,pazy}.

In the case where $X_0(t)$ is a L\'evy process started at zero and
$X(t)=x+X_0(t)$ for $x\in\rd$, the generator $L_x$ of the
semigroup $P(t)f(x)=E_x[f(X(t))]$ is a pseudo-differential
operator \cite{applebaum,Jacob,schilling} that can be explicitly
computed by inverting the L\'evy representation. The L\'{e}vy
process $X_0(t)$ has characteristic function
$$E[\exp(ik\cdot
X_0(t))]=\exp(t\psi(k))$$ with
$$
\psi (k)=ik\cdot a-\frac{1}{2}k\cdot Qk+ \int_{y\neq 0}\left(
e^{ik\cdot y}-1-\frac{ik\cdot y}{1+||y||^{2}}\right)\phi(dy),
$$
where $a\in \RR{R}^{d}$, $Q$ is a nonnegative definite matrix, and
$\phi$ is a $\sigma$-finite Borel measure on $\RR{R}^{d}$ such
that
$$
\int_{y\neq 0}\min \{1,||y||^{2}\}\phi(dy)<\infty;
$$
see, for example, \cite[Theorem 3.1.11]{RVbook} and \cite[Theorem
1.2.14]{applebaum}. Let
$$\hat{f}(k)=\int_{\RR{R}^{d}}e^{-ik\cdot
x}f( x)\,dx$$
denote the Fourier transform. Theorem 3.1 in
\cite{fracCauchy} shows that $L_x f(x)$ is the inverse Fourier
transform of $\psi(k)\hat{f}(k)$ for all $f\in D(L_x)$, where $D(L_x)$ is the domain of $L_x$ defined by
$$
D(L_x)=\{ f\in L^{1}(\RR{R}^{d}):\ \psi(k)\hat{f}(k)=\hat{h}(k)\
\exists \ h \in L^{1}(\RR{R}^{d})  \},
$$
and
\begin{equation}\begin{split}\label{pseudoDO}
L_x f(x)&= a\cdot\nabla f(x) +\frac{1}{2}\nabla \cdot Q\nabla
f(x)\\
&+ \int_{y \neq  0} \left(  f(x+y)-f(x)-\frac{\nabla
f(x)\cdot y}{1+y^{2}} \right)\phi(dy)
\end{split}\end{equation}
for all $f\in W^{2,1}(\RR{R}^{d})$, the Sobolev space of
$L^{1}$-functions whose first and  second partial derivatives are
all $L^{1}$-functions.  We can also view the genrator of a L\'evy process as
$L_x=\psi(-i\nabla)$ where $\nabla=(\partial/\partial
x_1,\ldots,\partial/\partial x_d)'$.  Some  examples are as follows.  When $\psi(k)=-D\|k\|^\alpha$ and
$L_x=-D(-\Delta)^{\alpha/2}$, a fractional derivative in space,
using the correspondence $k_{j}\to -i\partial/\partial x_j$ for
$1\leq j\leq d$ the corresponding process  $X_0(t)$ is
spherically symmetric stable. When $\psi(k)=D\sum_j (ik_j)^{\alpha_j}$ and
$L_x=D\sum_j \partial^{\alpha_j}/\partial x_j^\alpha$ using
Riemann-Liouville fractional derivatives in each variable, the corresponding process  $X_0$ has independent stable marginals. This
form is  different from  fractional Laplacian unless all
$\alpha_j=2$.

\section{ Main results}

Our first main result is the following theorem.

\begin{theorem}\label{general-thm-1}
Let $\alpha\in (0,1)$, and  $L_{x}$ be the generator of a continuous (Markov) semigroup
$P(t)f(x)(=E_{x}[f(X_{t})])$, and take $f\in D(L_{x})$, the domain of
the generator.  Then, the unique solution of  the equation with conformable time derivative
\begin{equation}
\begin{split}
T_\alpha u(t,x)&=L_x u(t,x), t>0\\
u(0,x)&=f(x),
\end{split}
\end{equation}
is given by
\begin{equation}
u(t,x) =P(t^\alpha/\alpha)f(x)(= E_{x}[f(X_{t^\alpha/\alpha})])
\end{equation}
\end{theorem}

\begin{proof}

The proof follows by a chain rule when the derivative of the semigroup, $t\to P(t)$ exists in the usual sense.
Let $u(t,x)=P(t^\alpha/\alpha)f(x)$. Then by chain rule
\begin{equation}\begin{split}
\frac{\partial}{\partial t}u(t,x)&=P^{'}(t^\alpha/\alpha)f(x) \frac{d t^\alpha/\alpha}{dt}\\
&=t^{\alpha-1}P^{'}(t^\alpha/\alpha)f(x)\\
&=t^{\alpha-1} L_x P(t^\alpha/\alpha)f(x)\\
&=t^{\alpha-1} L_x u(t,x).
\end{split}
\end{equation}

Hence  we have
$$
T_\alpha u(t,x)=t^{1-\alpha}\frac{\partial}{\partial t}u(t,x)=L_x u(t,x).
$$

Uniqueness holds since the solution of the  the equation \eqref{ACP} is unique and the solution of the conformable Cauchy problem is given in terms of the  solution of equation \eqref{ACP}.

\end{proof}

\begin{proof}[Proof of Theorem \ref{fbm-ctfd}]
This follows from the fact that the density of fractional Brownian motion, $p(t,x,y)=\frac{1}{(2\pi t^{2H})^{d/2}}\exp(-|x-y|^2/2t^{2H})$ satisfy
$$
\frac{\partial}{\partial t}p(t,x,y)=Ht^{2H-1}\Delta_x p(t,x,y)
$$
Hence we get
$$
T_{2H}p(t,x,y)=H\Delta_x p(t,x,y)
$$
Next we can apply the conformable fractional derivative and use the dominated convergence theorem  to the function 
$$u(t,x)=E[f(x+B^H(t))])=\int_\rd p(t,x,y)f(y)dy$$ to get
\begin{equation}\begin{split}
T_{2H}u(t,x)&=\int_\rd  T_{2H}p(t,x,y)f(y)dy\\
&= \int_\rd  H\Delta_x p(t,x,y)f(y)dy\\
&=H\Delta_x\int_\rd p(t,x,y)f(y)dy\\
&=H\Delta_x u(t,x).
\end{split}\end{equation}
\end{proof}

The next result is a restatement of Theorem \ref{general-thm-1} for
L\'evy semigroups.  {\bf The proof does not use the differentiability of} the semigroup $T(t)$, rather it relies on a Fourier
transform argument.  Recall the following notation for the
Fourier transform:
\begin{equation*}\begin{split}
\hat{u}(t,k)&=\int_{\RR{R}^{d}}e^{-ik\cdot x}u(t, x)dx ;
\end{split}\end{equation*}

\begin{theorem}\label{unified2}
Suppose that $X(t)=x+X_0(t)$ where $X_0(t)$ is a L\'evy process
starting at zero.  If $L_x$ is the generator \eqref{pseudoDO} of the
semigroup $P(t)f(x)=E_x[(f(X_t))]$ on $L^1(\rd)$, then, for any $f\in
D(L_{x})$, the unique solution of the following Cauchy problem
\begin{equation}\label{levy-type}
\begin{split}
T_\alpha u(t,x)&=L_x u(t,x), t>0\\
u(0,x)&=f(x),
\end{split}
\end{equation}
is given by
\begin{equation}
u(t,x) =P(t^\alpha/\alpha)f(x)(= E_{x}[f(X_{t^\alpha/\alpha})])
\end{equation}
\end{theorem}

\begin{proof}
Take Fourier transforms on both sides of \eqref{levy-type} to get
\begin{equation}\label{mmm1}
T_\alpha \hat{u}(t,k)= \psi (k)\hat{u}(t,k)
\end{equation}
using the fact that $\psi(k)\hat{f}(k)$ is the Fourier transform  of
$L_x f(x)$.
Then since the eigenfunctions of the conformable derivative operator
$$
T_\alpha g(t)=\lambda g(t); f(0)=c
$$
are  given by $g(t)=ce^{\frac{\lambda}{\alpha}t^{\alpha}}$, the solution of equation \eqref{mmm1} are given by
$$
\hat{u}(t,k)=ce^{\frac{\psi(k)}{\alpha}t^{\alpha}}= \hat{f}(k)e^{\frac{\psi(k)}{\alpha}t^{\alpha}}=\hat{f}(\mathrm{Characteristic\ \  function\ \  of}\ \ X_{t^\alpha/\alpha})
$$
since the initial function is $f(x)$ we have $\hat{u}(0,k)=\hat{f}(k)$, and  the Fourier transform (characteristic function) of the L\'evy process $X_t$ is $e^{t\psi(k)}$, we have by taking inverse Fourier transforms (this is given by the convolution of $f$ and the density of $X_{t^\alpha/\alpha}$)
$$
u(t,x)=\int_{\rd}p(t^\alpha/\alpha,x-y)f(y)dy=E_x[f(X_{t^\alpha/\alpha})]
$$

\end{proof}

Let $D\subset \rd$ be a bounded domain with a nice (smooth) boundary $\partial D$.
Let $p_D(t,x,y)$ denote the heat kernel of the equation
\begin{equation}\label{frac-laplacian-bounded-domain}\begin{split}
\frac{\partial}{\partial t}u(t,x)&=L^D_xu(t,x),\ \  x\in D, \ \ t>0\\
u(0,x)&=f(x),\ \ x\in D\\
u(t,x)&=0,\ \ x\in \partial D,\ \mathrm{when \ } \beta=2\\
u(t,x)&=0,\ \ x\in D^C,\ \mathrm{when \ } 0<\beta<2
\end{split}\end{equation}
here $L^D_x$ is the generator of the   stable L\'evy process of index $\beta\in (0, 2]$ killed at the first exit time, $\tau_D$, from the domain $D$. When $\beta=2$, $X_t$ is Brownian motion and $L_x^D=\Delta_x|_D$ is the Laplacian restricted to the bounded domain $D$. When $\beta\in (0,2)$, then $X_t$ is a stable L\' evy process and $L_x^D=-(\Delta)^{\beta/2}|_D$ is the fractional Laplacian restricted to the bounded domain $D$

There exist eigenvalues
$0< \lambda_1<\lambda_2\leq \lambda_3\cdots,$ such that $\lambda_n\to\infty,$ as
$n\to\infty$, with the  corresponding complete orthonormal set (in $H^2_0$ ) of
eigenfunctions $\phi_n$ of the operator $L_D$ satisfying

\begin{equation}\label{eigen-eigen}
 L^D_x \phi_n(x)=-\lambda_n \phi_n(x), \ x\in D;\  \\
\phi_n |_{\partial D}=0 \ \mathrm{when \ } \beta=2;\\
\phi_n |_{ D^C}=0 \ \mathrm{when \ } 0<\beta<2.
\end{equation}

 A well known fact is that
\begin{equation}\label{killed-kernel}
p_D(t,x,y)=\sum_{n=1}^\infty e^{-\lambda_n t}\phi_n(x)\phi_n(y)\ \ \mathrm{for \ all}\  x, y\ \in D, t>0.
\end{equation}
The series
converges absolutely and  uniformly on $[t_0,\infty)\times D\times D$ for all
$t_0>0$.

The unique solution of equation \eqref{frac-laplacian-bounded-domain} is given by

$$u(t,x)=\int_D p_D(t,x,y)f(y)dy=E_x[f(X_t)I(\tau_D>t)]=\sum_{n=1}^\infty\phi_n(x)e^{-\lambda_nt}\bigg[\int_D f(y)\phi_n(y)dy\bigg],$$ where here $X_t$ is a L\' evy process and $\tau_D=\inf\{s>0: X(s)\notin D\}$  is the first exit time of the process from $D$. See, for example, \cite{ cmn-12, mnv-09} for more on the killed stable process and its generator and semigroup.

Now we can state our theorem for the conformable Cauchy problem in bounded domains.
\begin{theorem}\label{unified2}
Suppose that $X(t)=x+X_0(t)$ where $X_0(t)$ is a stable L\'evy process of index $\beta\in (0, 2]$
starting at zero.  If $L^D_x$ is the generator \eqref{pseudoDO} of the
semigroup $P^D(t)f(x)=E_x[f(X_t)I(\tau_D>t)]$ on $H^2_0(D)$, then, for any $f\in
D(L^D_{x})$, the unique solution of the  following initial-boundary value problem
\begin{equation}\label{ bounded domain}
\begin{split}
T_\alpha u(t,x)&=L^D_x u(t,x), \ \ t>0,\ \ x\in D\\
u(0,x)&=f(x)\\
u(t,x)&=0,\ \ x\in \partial D,\ \mathrm{when \ } \beta=2\\
u(t,x)&=0,\ \ x\in D^C,\ \mathrm{when \ } 0<\beta<2
\end{split}
\end{equation}
is given by
\begin{equation}
\begin{split}
u(t,x)& =T^D(t^\alpha/\alpha)f(x)(= E_{x}[f(X_{t^\alpha/\alpha})I(\tau_D>t^\alpha/\alpha)])\\
&=\sum_{n=1}^\infty\phi_n(x)e^{-\lambda_nt^\alpha/\alpha}\bigg[\int_D f(y)\phi_n(y)dy\bigg]
\end{split}\end{equation}
\end{theorem}

\begin{proof}
The proof follows by a separation of variables as in \cite[section 2]{cenesiz-Kurt} and as in the proof of Theorem 5.1 in \cite{cmn-12}. We leave the details to the reader.
\end{proof}

When $D=[0,L]\subset \rr$ and $\beta=2$, then $\lambda_n=(n\pi/L)^2$, $\phi_n(x)=\sin (n\pi x/L)$. In this case we recover the result in \c Cenesiz and Kurt \cite[section 2]{cenesiz-Kurt}.


\end{document}